\documentclass[a4paper,10pt]{amsart}
\usepackage[all,cmtip]{xy}      
\usepackage{mathtools,amssymb,amsfonts,stmaryrd,amsmath,amsthm,enumerate,hyperref,color}
\usepackage[utf8]{inputenc}
\usepackage{enumitem}

\newtheorem{theorem}{Theorem}[section]
\newtheorem{lemma}[theorem]{Lemma}

\newtheorem{corollary}[theorem]{Corollary}

\newtheorem{theorema}{Theorem}

\theoremstyle{definition}
\newtheorem{definition}[theorem]{Definition}
\newtheorem{remark}[theorem]{Remark}

\newtheorem{question}[theorem]{Question}

\theoremstyle{remark}

\begin{document}

\title[Product decompositions and $B$-rigidity]{Product decompositions of moment-angle manifolds and $B$-rigidity
} 

\author[S.\ Amelotte]{Steven Amelotte}
\address{Institute for Computational and Experimental Research in Mathematics, Brown University, Providence, RI 02903, U.S.A.}
\email{steven\_amelotte@brown.edu} 

\author[B.\ Briggs]{Benjamin Briggs}
\address{Mathematical Sciences Research Institute, 17 Gauss Way, Berkeley, CA 94720, U.S.A.}
\email{bbriggs@msri.org}


\thanks{During this work the first author was hosted by the Institute for Computational and Experimental Research in Mathematics in Providence, RI, supported by the National Science Foundation under Grant No.\ 1929284.
The second author was hosted by the Mathematical Sciences Research Institute in Berkeley, California,
  supported by the National Science Foundation under Grant No.\ 1928930}

\begin{abstract}
A simple polytope $P$ is called \emph{B-rigid} if its combinatorial type is determined by the cohomology ring of the moment-angle manifold $\mathcal{Z}_P$ over $P$. We show that any tensor product decomposition of this cohomology ring is geometrically realized by a product decomposition of the moment-angle manifold up to equivariant diffeomorphism. As an application, we find that $B$-rigid polytopes are closed under products, generalizing some recent results in the toric topology literature. Algebraically, our proof establishes that the Koszul homology of a Gorenstein Stanley--Reisner ring admits a nontrivial tensor product decomposition if and only if the underlying simplicial complex decomposes as a join of full subcomplexes.
\end{abstract}

\maketitle

\section{Introduction}

Problems surrounding cohomological rigidity have received a great deal of attention throughout the development of toric topology; for a recent overview, we recommend \cite{FMW20} and references therein. Recall that a family $\mathcal{C}$ of spaces (or of smooth manifolds) is called \emph{cohomologically rigid} if for all $X$, $Y\in \mathcal{C}$, a ring isomorphism $H^\ast(X)\cong H^\ast(Y)$ implies that $X$ and $Y$ are homeomorphic (resp.\ diffeomorphic). Toric topology associates to each simplicial complex $K$ with $m$ vertices the \emph{moment-angle complex} $\mathcal{Z}_K$, a finite CW-complex with an action of the torus $T^m=(S^1)^m$, whose equivariant topology is intimately tied to the combinatorics of $K$ and the homological algebra of its Stanley--Reisner ring. For example, if $K=\partial P^*$ for some simple polytope $P$, then $\mathcal{Z}_P\coloneqq \mathcal{Z}_K$ has the structure of a smooth $T^m$-manifold. More generally, when the Stanley--Reisner ring $\mathbb{Z}[K]$ is Gorenstein, $\mathcal{Z}_K$ is a topological manifold called a \emph{moment-angle manifold}. Computational evidence and a lack of counterexamples (despite the attention of many authors) makes a positive solution to the cohomological rigidity problem for this family of manifolds plausible.

An individual moment-angle manifold is said to be \emph{cohomologically rigid} if its homeomorphism type among all moment-angle manifolds is distinguished by its cohomology ring. Since the combinatorial type of $K$ determines $\mathcal{Z}_K$ up to equivariant homeomorphism, one way to produce cohomologically rigid moment-angle manifolds is to find simplicial complexes $K$ whose combinatorial type is uniquely determined by the ring $H^\ast(\mathcal{Z}_K)$.

\begin{question} \label{B_question}
Let $K_1$ and $K_2$ be simplicial complexes and let $k$ be a field. When does a graded ring isomorphism $H^\ast(\mathcal{Z}_{K_1};k)\cong H^\ast(\mathcal{Z}_{K_2};k)$ imply a combinatorial equivalence $K_1\simeq K_2$?
\end{question}

The question above is due to Buchstaber \cite{B}, and simplicial complexes or simple polytopes for which the answer is positive are known as \emph{B-rigid} (see Section~\ref{rigidity_sec} for precise definitions). 
Question~\ref{B_question} is reduced to a commutative algebraic problem by (1) the Stanley--Reisner correspondence between simplicial complexes and square-free monomial rings, and (2) an identification of $H^\ast(\mathcal{Z}_K;k)$ with the Tor-algebra of the Stanley--Reisner ring $k[K]$ (see Section~\ref{prelim_sec}).

$B$-rigid complexes seem to be rare in general (and indeed, combinatorially distinct polytopes often define diffeomorphic moment-angle manifolds, see \cite[Example~3.4]{BEMPP}), but examples to date have proven useful in establishing a variety of rigidity results for moment-angle manifolds, (quasi)toric manifolds, small covers and other related spaces. 

For example, flag simplicial $2$-spheres without chordless cycles of length $4$ were shown to be $B$-rigid by Fan, Ma and Wang \cite{FMW15}, implying the cohomological rigidity of a large class of moment-angle manifolds including $\mathcal{Z}_P$ for all \emph{fullerenes} $P$ (i.e. simple $3$-polytopes with only pentagonal and hexagonal facets). More generally, this class of simplicial $2$-spheres contains the dual simplicial complexes $\partial P^*$ of all \emph{Pogorelov} polyopes $P$, and their $B$-rigidity was used by Buchstaber et al.\ \cite{BEMPP} to obtain cohomological rigidity results for all quasitoric manifolds and small covers over these $3$-polytopes.

Finite products of simplices are $B$-rigid by a result of Choi, Panov and Suh~\cite{CPS}. This was used to show any quasitoric manifold $M$ with $H^\ast(M)\cong H^\ast(\prod_{i=1}^\ell \mathbb{C}P^{n_i})$ is in fact homeomorphic to $\prod_{i=1}^\ell \mathbb{C}P^{n_i}$ \cite[Theorem~1.5]{CPS}. Quasitoric manifolds over products of simplices include an important family of toric varieties known as \emph{(generalized) Bott manifolds}, and here $B$-rigidity implies that if a quasitoric manifold has cohomology isomorphic to that of a generalized Bott manifold, then their orbit spaces are combinatorially equivalent.

The $B$-rigidity of a prism (i.e.\ a product of a polygon and an interval) follows from \cite{CK}, and this result was generalized to the product of a polygon and any simplex by Choi and Park in their study of projective bundles over toric surfaces \cite{CP}. We will return to this example in Section~\ref{quasitoric_sec}.

These last examples motivate the question of whether $B$-rigid polytopes are closed under products. This question is also raised in work of Bosio~\cite{Bo} where a weaker property for polytopes, \emph{puzzle-rigidity}, is shown to be closed under products. The purpose of this short note is to answer this question affirmatively.

\begin{theorema}\label{thmA}
The collection of B-rigid Gorenstein complexes is closed under finite joins. Consequently, the collection of B-rigid polytopes is closed under finite products.
\end{theorema}

This result substantially extends the class of known $B$-rigid polytopes, and subsumes those products  established to be $B$-rigid  in \cite{CK,CPS,CP}. We deduce Theorem~\ref{thmA} from the following more general structure theorem for moment-angle manifolds.

\begin{theorema}\label{thmB}
Let $K$ be a simplicial complex on $m$ vertices and let $k$ be a field. If $K$ is Gorenstein over $k$ and the cohomology ring of $\mathcal{Z}_K$ admits a tensor product decomposition 
\[
H^\ast(\mathcal{Z}_K;k) \cong A_1\otimes\cdots\otimes A_\ell 
\] 
as graded $k$-algebras, then there is a $T^m$-equivariant homeomorphism
\[
\mathcal{Z}_K \cong \mathcal{Z}_{K_1}\times\cdots\times \mathcal{Z}_{K_\ell} 
\] 
with $H^\ast(\mathcal{Z}_{K_i};k) \cong A_i$ for each $i=1,\ldots,\ell$. 
\end{theorema}

The remainder of the paper is organized as follows. In Section~\ref{prelim_sec} we review the relevant background on simplicial complexes, Stanley--Reisner rings and moment-angle complexes and collect some facts concerning their cohomology rings. In Section~\ref{rigidity_sec} we define $B$-rigidity and prove Theorems~\ref{thmA} and \ref{thmB}. In Section~\ref{quasitoric_sec} we consider some consequences of the results above for quasitoric manifolds. The cohomology rings of these spaces are given by artinian reductions of Stanley--Reisner rings, and the main results of this section (Theorem~\ref{quasitoric_thm1} and Theorem~\ref{quasitoric_thm2}) address a variant of Question~\ref{B_question} that asks to what extent the orbit polytope of a quasitoric manifold $M$ is determined by $H^\ast(M)$.

\section{Preliminaries} \label{prelim_sec}

Let $k$ be a field. All graded algebras considered in this paper will be connected (that is, nonnegatively graded with $A^0=k$) and of finite type. The term \emph{commutative graded algebra} will refer to a graded algebra that is commutative in the graded sense. All tensor products are taken over $k$.

\subsection{Simplicial complexes and Stanley--Reisner rings}\label{simp_comp_sect}
Let $K$ be an abstract simplicial complex on the vertex set $[m]=\{1,\ldots,m\}$. We will always assume that $\varnothing\in K$ and that $\{i\}\in K$ for all $i\in [m]$. 

For a subset $I\subseteq [m]$, the \emph{full subcomplex} of $K$ on $I$ is defined by
\[
K_I = \{\sigma\in K \,:\, \sigma\subseteq I\}.
\]

A \emph{missing face} of $K$ is a minimal non-face, that is, a subset $I\subseteq [m]$ with $I\notin K$ and $J\in K$ for every proper subset $J\subset I$ (equivalently, $K_I=\partial\Delta^{|I|-1}$). We will write ${\rm MF}(K)$ for the set of missing faces of $K$.

The \emph{join} of two simplicial complexes $K_1$ and $K_2$ on disjoint vertex sets is defined to be the simplicial complex
\[
K_1 \ast K_2 = \{\sigma\cup\tau \,:\, \sigma\in K_1,\, \tau\in K_2\}. 
\]
In particular, $K\ast \Delta^0$ is called the \emph{cone} over $K$. Note that a simplicial complex on a given vertex set $[m]$ is uniquely determined by its set of missing faces and that $K$ decomposes as $K=K_I\ast K_J$ if and only if the vertex set admits a partition $[m]=I\sqcup J$ with the property that ${\rm MF}(K)={\rm MF}(K_I) \sqcup {\rm MF}(K_J)$.

The polynomial algebra $S=k[v_1,\ldots,v_m]$, having variables in bijection with the vertices of $K$, is naturally graded by $\mathbb{Z}^m$. We refer to this grading as the \emph{multidegree} and use the notation $\operatorname{mdeg}(v_1^{e_1}\!\cdots v_m^{e_m})=(e_1,\ldots,e_m)$.

The \emph{Stanley--Reisner ring} of $K$ is the multigraded algebra
\[
k[K]=S/\big(v_{i_1}\!\cdots v_{i_t} \,:\, \{i_1,\ldots,i_t\} \in {\rm MF
}(K) \big).
\]
The Koszul complex of the Stanley--Reisner ring
is the differential graded algebra
\begin{equation} \label{Koszul_comp}
\operatorname{Kos}^{{k[K]}}(v_1,\ldots,v_m) = \big( k[K]\otimes \Lambda(u_1,\ldots,u_m),\ d\big), \quad d(u_i)=v_i,
\end{equation}
with its usual exterior algebra product and each $u_i$ in homological degree $1$. 
The Tor-algebra
\[
\operatorname{Tor}^S_*(k[K],k) ={ H}_*\big(\operatorname{Kos}^{{k[K]}}(v_1,\ldots,v_m)\big)
\]
inherits its product from the Koszul complex. It is naturally graded by $\mathbb{Z}\times \mathbb{Z}^m$; we write $\operatorname{Tor}^S_i(k[K],k)_J$ for the subspace of homological degree $i$ and multidegree $J$.

A multidegree $J\in \mathbb{Z}^m$ is said to be \emph{square-free} if each of its entries is $0$ or $1$. The set of square-free multidegrees can be identified with the set of subsets of $[m]$, and we make this identification in the next result, which expresses $\operatorname{Tor}^S_*(k[K],k)$ in terms of the reduced simplicial cohomology groups of full subcomplexes of $K$.

\begin{theorem}[Hochster's formula \cite{Hochster}] \label{Hoch_thm} 
If $J$ is square-free, then
\[
\operatorname{Tor}^S_i(k[K],k)_J\cong \widetilde{H}^{|J|-i-1}(K_J;k),
\]
with the convention that $\widetilde{H}^{-1}(K_\varnothing;k)=k$. Otherwise $\operatorname{Tor}^S_i(k[K],k)_J=0$.
\end{theorem}

\begin{remark}\label{rem_sqfree}
Hochster's formula allows us to identify multidegrees of $\operatorname{Tor}^S_\ast(k[K],k)$ with subsets of $[m]$, and we will do this henceforth. It follows as well that for any subset $J\subseteq [m]$ the graded subspace \[
\operatorname{Tor}^S_\ast(k[K],k)_{\subseteq J}= \bigoplus_{I\subseteq J}\operatorname{Tor}^S_\ast(k[K],k)_I
\]
is a subalgebra of $\operatorname{Tor}^S_\ast(k[K],k)$, and moreover that the natural projection
\[
\mathrm{proj}_{\subseteq J}\colon \operatorname{Tor}^S_\ast(k[K],k) \longrightarrow \operatorname{Tor}^S_\ast(k[K],k)_{\subseteq J}
\]
is an algebra homomorphism. 
\end{remark}

\begin{lemma} \label{MF_lem}
In homological degree $1$, a $k$-basis for $\operatorname{Tor}_{1}^S(k[K],k)$ is given by 
\[
\left\{ [v_{i_1}\!\cdots v_{i_t}u_{i_{t+1}}]  \,:\, \{i_1,\ldots,i_{t+1}\} \in {\rm MF
}(K) \right\}.
\] 
\end{lemma}

\begin{proof}
A direct computation with the Koszul complex shows that $v_{i_1}\!\cdots v_{i_t}u_{i_{t+1}}$ is closed and not exact when $\{i_1,\ldots,i_{t+1}\} \in {\rm MF}(K)$. The resulting nonzero classes $[v_{i_1}\!\cdots v_{i_t}u_{i_{t+1}}]$ are linearly independent since they lie in distinct multidegrees, and therefore form a basis since $\dim_k\operatorname{Tor}_{1}^S(k[K],k)=|{\rm MF}(K)|$, the minimal number of  generators of the Stanley--Reisner ideal.
\end{proof}

\subsection{Moment-angle complexes and their cohomology rings}
We now recall some notions from toric topology, largely in order to lay out our notation. Fix as before a simplicial complex $K$ on a vertex set  $[m]$. Identify $D^2$ and $S^1$ with the unit disk in $\mathbb{C}$ and its boundary, respectively. 

The \emph{moment-angle complex} over $K$ is the subspace of the polydisk $(D^2)^m$ defined by
\[
\mathcal{Z}_K = \bigcup_{\sigma\in K} (D^2,S^1)^\sigma \subseteq (D^2)^m
\]
where
\[
(D^2,S^1)^\sigma=\left\{ (z_1,\ldots,z_m)\in (D^2)^m \,:\, z_i\in S^1\text{ if }i\notin \sigma\right\}.
\]
The coordinatewise action of the $m$-torus  $T^m=(S^1)^m$ on $(D^2)^m$ restricts to an action of $T^m$ on $\mathcal{Z}_K$. The equivariant cohomology ring $H^\ast_{T^m}(\mathcal{Z}_K;k)$ is naturally a graded module over $H^\ast(BT^m;k)\cong S$ (with polynomial generators $v_1,\ldots,v_m$ in cohomological degree $2$), and there is an isomorphism of graded $S$-modules 
\[
H^\ast_{T^m}(\mathcal{Z}_K;k) \cong k[K],
\]
see for example \cite[Corollary~4.3.3]{BP}.

For the ordinary cohomology, the cellular cochain complex of $\mathcal{Z}_K$ can be shown to be quasi-isomorphic to the Koszul complex \eqref{Koszul_comp}, which leads to the following fundamental result.

\begin{theorem}[{\cite[Theorem~4.5.4]{BP}}] \label{t_H_Tor}
There is an isomorphism of graded algebras
\[
H^\ast(\mathcal{Z}_K;k) \cong \operatorname{Tor}^S_\ast(k[K],k)
\]
where
\[
H^n(\mathcal{Z}_K;k) \cong \bigoplus_{n=2|J|-i} \operatorname{Tor}^S_i(k[K],k)_J.
\]
\end{theorem}

In particular, the cohomology ring of $\mathcal{Z}_K$ obtains a multigrading from the Tor-algebra. Under the isomorphism of Theorem~\ref{t_H_Tor}, the projection operators
\[
\mathrm{proj}_{\subseteq J}\colon \operatorname{Tor}^S_\ast(k[K],k) \longrightarrow \operatorname{Tor}^S_\ast(k[K],k)_{\subseteq J}
\]
of Remark~\ref{rem_sqfree} can be identified with homomorphisms
\[ 
j^\ast\colon H^\ast(\mathcal{Z}_K;k) \longrightarrow H^\ast(\mathcal{Z}_{K_J};k)
\]
induced by the natural maps $j\colon\mathcal{Z}_{K_J} \to \mathcal{Z}_K$ coming from inclusions of full subcomplexes $K_J \to K$. 

\begin{definition}
A commutative graded $k$-algebra $A$ is a \emph{Poincar\'e duality algebra} of dimension $n$ if it is finite dimensional, and if the bilinear forms \[ A^i\otimes A^{n-i} \to k, \quad a\otimes b \mapsto \epsilon(ab) \]
are nondegenerate for some homogeneous map $\epsilon\colon A \to k$ of degree $-n$.
\end{definition}

\begin{remark}
A finite dimensional, connected, graded algebra $A$ is Poincar\'e duality of dimension $n$ if and only if the socle $\mathrm{soc}(A)=\mathrm{ann}_A(A_{\geqslant 1})$ is one dimensional and concentrated in degree $n$. In this case, the orientation map $\epsilon\colon A\to k$ corresponds to a choice of isomorphism $\mathrm{soc}(A)\to k$.
\end{remark}

\begin{lemma} \label{PD_factors}
Let $A$ be a Poincar\'e duality algebra of dimension $n$. If there is an isomorphism $A \cong A_1\otimes A_2$ of graded algebras, then $A_1$ and $A_2$ are Poincar\'e duality algebras of dimension $n_1$ and $n_2$, respectively, with $n_1+n_2=n$.
\end{lemma}

\begin{proof}
Since $A \cong A_1\otimes A_2$ implies $\mathrm{soc}(A)\cong \mathrm{soc}(A_1)\otimes \mathrm{soc}(A_2)$, the socle of $A$ is one dimensional if and only if so are the socles of  $A_1$ and $A_2$.
\end{proof}

A classical result of Avramov and Golod \cite{AG} characterizes Poincar\'e duality for the Koszul homology of a local ring in terms of the Gorenstein property. The simplicial complex $K$ is called \emph{Gorenstein over} $k$ if the Stanley--Reisner ring $k[K]$ is a Gorenstein ring, and $K$ is simply called \emph{Gorenstein} if it is Gorenstein over every field $k$. Moreover, $K$ is called \emph{Gorenstein}${}^*$ if it is Gorenstein and $K$ is not a cone (i.e. $K=\mathrm{core}(K)$).

Combined with Theorem~\ref{t_H_Tor}, a graded version of Avramov--Golod's theorem (see \cite[Theorem~3.4.5]{BH}) leads to the following characterization of Poincar\'e duality for moment-angle complexes.

\begin{theorem}[{\cite[Theorem~4.6.8]{BP}}] \label{PD_thm}
The cohomology ring $H^\ast(\mathcal{Z}_K;k)$ is a Poincar\'e duality algebra if and only if $K$ is Gorenstein over $k$.
\end{theorem}

\begin{remark}
The topology behind this characterization is clarified by a result of Cai which states that $\mathcal{Z}_K$ is a closed topological $(m+n)$-manifold precisely when $K$ is a generalized homology $(n-1)$-sphere \cite[Corollary~2.10]{Cai}, which in turn is equivalent to the condition that $K$ is Gorenstein${}^*$ by a well-known result of Stanley \cite{S}.
\end{remark}

Gorenstein${}^*$ complexes of particular importance in toric topology arise from simple polytopes as follows.
Let~$P$ be a simple (convex) polytope of dimension $n$, and denote the facets (faces of dimension $n-1$) of $P$ by $F_1,\ldots,F_m$. The \emph{simplicial complex $K$ dual to} $P$ has vertex set $[m]=\{1,\ldots,m\}$ and $J\subseteq [m]$ belongs to $K$ if and only if $\bigcap_{i\in J}F_i\neq \varnothing$. Equivalently, $K$ is the boundary of the polytope dual to $P$, i.e.\ $K=\partial P^*$.

If $K'$ is the simplicial complex dual to another simple polytope $P'$, then it is easy to see that the simplicial complex dual to the product $P\times P'$ is the join $K\ast K'$.

Through this construction we associate to $P$ the Stanley--Reisner ring and moment-angle manifold 
\[
k[P]=k[K] \quad\textrm{ and }\quad \mathcal{Z}_P=\mathcal{Z}_K.
\]
In particular, if $k[P]$ is generated by $v_1,\ldots,v_m$ as in Section~\ref{simp_comp_sect}, then by Theorem~\ref{t_H_Tor}
\[
H^*(\mathcal{Z}_P;k)\cong H^*\big(\operatorname{Kos}^{{k[P]}}(v_1,\ldots,v_m)\big).
\]
\noindent Moment-angle manifolds corresponding to simple polytopes can be given smooth structures (for which the $T^m$-action is smooth) by identifying $\mathcal{Z}_P$ with a nonsingular intersection of quadrics in $\mathbb{C}^m$ (see \cite[Section~6.1]{BP} for details).

\section{$B$-rigidity}\label{rigidity_sec}

In this section we prove Theorem~\ref{thmB} (Theorem~\ref{main_decomp_th} below) and use it to derive some corollaries, including Theorem~\ref{thmA} (Corollary~\ref{product_cor} below). We begin with a definition.

A simplicial complex $K$ with $K=\mathrm{core}(K)$ is called \emph{B-rigid over} $k$ if for every $K'$ with $K'=\mathrm{core}(K')$, a graded ring isomorphism $H^\ast(\mathcal{Z}_K;k)\cong H^\ast(\mathcal{Z}_{K'};k)$ implies a combinatorial equivalence $K\simeq K'$. We simply say that $K$ is \emph{B-rigid} if it is $B$-rigid over every field $k$.

\begin{remark}
$B$-rigidity was introduced and defined as above in \cite{CPS} to address Question~\ref{B_question} (cf.\ \cite[Lecture~IV, Problem~7.6]{B}), and has since been studied in \cite{BEMPP,CP19,E,FMW15,FMW20}, where some variations on the definition above have appeared. For example, in \cite{FMW20}, bigraded isomorphisms of cohomology rings are used to define $B$-rigidity while complexes satisfying the definition above are called \emph{strongly B-rigid}. In \cite{FMW15}, a complex is called $B$-rigid when it is $B$-rigid over $\mathbb{Z}$ in the sense above. We note that any complex satisfying the definition above is also $B$-rigid in the bigraded and integral senses.

In \cite{BEMPP,CP19,E}, the notion of $B$-rigidity is defined for simple polytopes rather than simplicial complexes. A simple polytope $P$ is called \emph{B-rigid over} $k$ if for every simple polytope $P'$, a graded ring isomorphism $H^\ast(\mathcal{Z}_P;k)\cong H^\ast(\mathcal{Z}_{P'};k)$ implies a combinatorial equivalence $P\simeq P'$. As before we say that $P$ is \emph{B-rigid} if it is $B$-rigid over every field $k$. Our main results apply in the more general context of simplicial complexes.
\end{remark}

\begin{theorem}\label{main_decomp_th}
Let $K$ be a simplicial complex on vertex set $[m]$ and let $k$ be a field. If $K$ is Gorenstein over $k$ and the cohomology ring of $\mathcal{Z}_K$ admits a tensor product decomposition 
\[
H^\ast(\mathcal{Z}_K;k) \cong A_1\otimes\cdots\otimes A_\ell 
\] 
as graded $k$-algebras, then there is a $T^m$-equivariant homeomorphism \[ \mathcal{Z}_K \cong \mathcal{Z}_{K_1}\times\cdots\times \mathcal{Z}_{K_\ell} \] where $H^\ast(\mathcal{Z}_{K_i};k) \cong A_i$ for each $i=1,\ldots,\ell$. 
\end{theorem}

\begin{proof}
Note that if $K=K'\ast\Delta^0$, then $\mathcal{Z}_K$ is $T^m$-equivariantly homeomorphic to $\mathcal{Z}_{K'} \times \mathcal{Z}_{\Delta^0}=\mathcal{Z}_{K'} \times (D^2)$ with $T^m=T^{m-1}\times S^1$ acting coordinatewise on the product. We may therefore assume without loss of generality that $K$ is not a cone and hence that $K$ is Gorenstein${}^*$ over $k$.

It will suffice to assume $\ell=2$, so suppose \[ \varphi\colon A_1\otimes A_2 \longrightarrow H^\ast(\mathcal{Z}_K;k) \] is an isomorphism of graded $k$-algebras. Since $K$ is Gorenstein${}^*$ over $k$, the cohomology ring $H^\ast(\mathcal{Z}_K;k)$ is a Poincar\'e duality algebra by Theorem~\ref{PD_thm}, and so are $A_1$ and $A_2$ by Lemma~\ref{PD_factors}. Let $\tau_i \in A_i$ be a generator of the highest degree nonzero graded component of $A_i$, $i=1,2$. Then $\tau \coloneqq \varphi(\tau_1) \cdot \varphi(\tau_2)$ generates the top cohomology group of $\mathcal{Z}_K$, and it follows from the Gorenstein${}^*$ property that $\tau$ is homogeneous of multidegree $\operatorname{mdeg}(\tau)=[m]$. Identifying $H^\ast(\mathcal{Z}_K;k)$ with 
$\bigoplus_{J\subseteq [m]} \mathrm{Tor}_\ast^S(k[K],k)_J$ using Theorem~\ref{t_H_Tor},
write 
\[ 
\varphi(\tau_i) = {\textstyle \sum_{J\subseteq [m]}} \varphi(\tau_i)_J, \quad  \varphi(\tau_i)_J\in \mathrm{Tor}_\ast^S(k[K],k)_J
\]
for $i=1,2$ and choose multidegrees $U,V\subseteq [m]$ such that $\varphi(\tau_1)_U \cdot \varphi(\tau_2)_V = c\tau$ for some nonzero $c\in k$. Note that $U\cap V=\varnothing$ by square-freeness (Theorem~\ref{Hoch_thm}), and $U\cup V=[m]$ since $\operatorname{mdeg}(\tau)=[m]$.

Next, define subalgebras of $H^\ast(\mathcal{Z}_K;k)$ by
\[
\widetilde{A}_1 = \mathrm{proj}_{\subseteq U}(\varphi(A_1))\quad\text{and}\quad
\widetilde{A}_2 = \mathrm{proj}_{\subseteq V}(\varphi(A_2))
\]
using the projections from Remark~\ref{rem_sqfree}. Since $\mathrm{proj}_{\subseteq U} \varphi(\tau_1) \neq 0$, and every nonzero element of $\varphi(A_1)$ divides $\varphi(\tau_1)$ by Poincar\'e duality, it follows that $\mathrm{proj}_{\subseteq U}|_{\varphi(A_1)}$ is injective and therefore defines an isomorphism $\varphi(A_1) \cong \widetilde{A}_1$. Similarly, $\varphi(A_2) \cong \widetilde{A}_2$. Consider the homomorphism of graded algebras $\widetilde{A}_1 \otimes \widetilde{A}_2 \to H^\ast(\mathcal{Z}_K;k)$ defined by multiplication. Since $\widetilde{A}_1 \otimes \widetilde{A}_2$ is clearly a Poincar\'e duality algebra with socle generated by $\varphi(\tau_1)_U \cdot \varphi(\tau_2)_V$, this map is injective and therefore an isomorphism as $\dim_k(\widetilde{A}_1 \otimes \widetilde{A}_2) =\dim_k(A_1 \otimes A_2) =\dim_k H^\ast(\mathcal{Z}_K;k)$.

Finally, in homological degree $1$, observe that the decomposition $H^\ast(\mathcal{Z}_K;k) \cong \widetilde{A}_1 \otimes \widetilde{A}_2$ implies that each indecomposable generator 
\[
[v_{i_1}\!\cdots v_{i_t}u_{i_{t+1}}] \in \mathrm{Tor}_{1}^S(k[K],k) \subseteq H^\ast(\mathcal{Z}_K;k)
\]
of Lemma~\ref{MF_lem} must lie either in $\widetilde{A}_1$ or $\widetilde{A}_2$ (since it is homogeneous of multidegree $\{i_1,\ldots,i_{t+1}\}$ and $\widetilde{A}_1 \subseteq \operatorname{Tor}^S_\ast(k[K],k)_{\subseteq U}$, $\widetilde{A}_2 \subseteq \operatorname{Tor}^S_\ast(k[K],k)_{\subseteq V}$). In particular, each missing face $\sigma=\{i_1,\ldots,i_{t+1}\} \notin K$ satisfies either $\sigma\in U$ or $\sigma\in V$, that is, $\mathrm{MF}(K)=\mathrm{MF}(K_U) \sqcup \mathrm{MF}(K_V)$. It follows that $K=K_U\ast K_V$ and hence $\mathcal{Z}_K$ is equivariantly homeomorphic to $\mathcal{Z}_{K_U} \times \mathcal{Z}_{K_V}$. Note that since $H^\ast(\mathcal{Z}_{K_U};k)\cong \operatorname{Tor}^S_\ast(k[K],k)_{\subseteq U}$ and $H^\ast(\mathcal{Z}_{K_V};k)\cong \operatorname{Tor}^S_\ast(k[K],k)_{\subseteq V}$, both inclusions $\widetilde{A}_1 \subseteq \operatorname{Tor}^S_\ast(k[K],k)_{\subseteq U}$ and $\widetilde{A}_2 \subseteq \operatorname{Tor}^S_\ast(k[K],k)_{\subseteq V}$ must be equalities for dimension reasons. Therefore 
\[
H^\ast(\mathcal{Z}_{K_U};k)\cong \widetilde{A}_1 \cong \varphi(A_1)\cong A_1,
\]
and similarly $H^\ast(\mathcal{Z}_{K_V};k)\cong A_2$, as desired.
\end{proof}

The next corollary follows immediately from the proof above. It is interesting to note that any product decomposition of a moment-angle manifold up to homotopy can be improved to one up to $T^m$-equivariant homeomorphism.

\begin{corollary}\label{tfae_cor}
Let $K$ be a Gorenstein simplicial complex on the vertex set $[m]$. The following conditions are equivalent:
\begin{enumerate}[label={\normalfont(\alph*)}]
    \item $H^\ast(\mathcal{Z}_K)$ decomposes nontrivially as a tensor product;
    \item $\mathcal{Z}_K$ is homotopy equivalent to a product of non-contractible spaces;
    \item $\mathcal{Z}_K$ is $T^m$-equivariantly homeomorphic to a product of non-contractible moment-angle manifolds;
    \item $K=K_1\ast K_2$, where $K_i$ is not a simplex, $i=1,2$.
\end{enumerate}
\end{corollary}

\begin{remark}
Of course, for a moment-angle manifold over a simple polytope $P$, it follows from the results above that a nontrivial tensor product decomposition of $H^\ast(\mathcal{Z}_P;k)$ implies a product decomposition of $P$. In this case the $T^m$-equivariant homeomorphisms of Theorem~\ref{main_decomp_th} and Corollary~\ref{tfae_cor} can be replaced by $T^m$-equivariant diffeomorphisms.
\end{remark}

\begin{corollary}\label{product_cor}
The collection of B-rigid Gorenstein complexes is closed under finite joins. Consequently, the collection of B-rigid polytopes is closed under finite products.
\end{corollary}

\begin{proof}
Let $K_1$ and $K_2$ be $B$-rigid Gorenstein complexes. To see that $K_1\ast K_2$ is $B$-rigid, assume there is a graded ring isomorphism $H^\ast(\mathcal{Z}_{K_1\ast K_2};k)\cong H^\ast(\mathcal{Z}_L;k)$ for some complex $L$ with $L=\mathrm{core}(L)$. Then since $H^\ast(\mathcal{Z}_{K_1\ast K_2};k)\cong H^\ast(\mathcal{Z}_{K_1};k)\otimes H^\ast(\mathcal{Z}_{K_2};k)$, the proof of Theorem~\ref{main_decomp_th} implies that $L$ decomposes as a join $L_1\ast L_2$ with $H^\ast(\mathcal{Z}_{L_i};k)\cong H^\ast(\mathcal{Z}_{K_i};k)$ for $i=1,2$. It now follows from the $B$-rigidity of $K_i$ that there is a combinatorial equivalence $K_i\simeq L_i$ for $i=1,2$, and hence $L\simeq K_1\ast K_2$. Therefore $K_1\ast K_2$ is $B$-rigid.
\end{proof}

\section{Quasitoric manifolds} \label{quasitoric_sec}

In their foundational work \cite{DJ}, Davis and Januszkiewicz introduced the notion of a quasitoric manifold as a topological generalization of a nonsingular projective toric variety (or toric manifold). In this section we discuss some implications of the results of the previous section for quasitoric manifolds.

Let $P$ be a simple convex polytope of dimension $n$. A \emph{quasitoric manifold} over $P$ is a closed smooth $2n$-dimensional manifold $M$ equipped with a smooth locally standard action of $T^n$ such that the orbit space $M/T^n$ can be identified with $P$. (A \emph{locally standard} $T^n$-action is one which is locally modeled on the standard action of $T^n$ on $\mathbb{C}^n$, and this property implies that the orbit space is a manifold with corners.) 

Every quasitoric manifold over $P$ arises as a quotient $M\cong \mathcal{Z}_P/T^{m-n}$ for some subtorus $T^{m-n}\subseteq T^m$ that acts freely on $\mathcal{Z}_P$ \cite[Proposition~7.3.12]{BP}, 
resulting in a principal fibration sequence 
\[
T^{m-n} \longrightarrow \mathcal{Z}_P \longrightarrow M.
\]

The cohomology of quasitoric manifolds was described in \cite{DJ}. Associated to the $T^n$-action on $M$, there is a regular sequence $t_1,\ldots,t_{n}$ of linear elements in $k[P]$ such that 
\begin{equation} \label{M_eq}
H^\ast(M;k)\cong k[P]/(t_1,\ldots,t_{n}). 
\end{equation}
Since $t_1,\ldots,t_n$ is a regular sequence, it follows from Theorem~\ref{t_H_Tor} that there is an isomorphism (cf.\ \cite[Lemma~A.3.5]{BP})
\begin{equation} \label{tor_eq}
H^\ast(\mathcal{Z}_P;k)\cong \operatorname{Tor}^{S/\mathcal{J}}_\ast(k[P]/\mathcal{J},k), 
\end{equation}
where $\mathcal{J}$ denotes the ideal generated by (homogeneous lifts to $S$ of) $t_1,\ldots,t_n$.

As a consequence of the isomorphisms \eqref{M_eq} and \eqref{tor_eq}, we emphasize that the cohomology ring of any quasitoric manifold over $P$ determines, in particular, the cohomology ring of the moment-angle manifold over $P$.

\begin{lemma}[{\cite[Lemma~3.7]{CPS}}] \label{M_Z_lem}
Let $M$ and $M'$ be quasitoric manifolds over simple polytopes $P$ and $P'$, respectively. If there is a graded ring isomorphism $H^\ast(M;k)\cong H^\ast(M';k)$, then there is a graded ring isomorphism $H^\ast(\mathcal{Z}_P;k)\cong H^\ast(\mathcal{Z}_{P'};k)$.
\end{lemma}

\begin{remark} 
The cohomology ring of a quasitoric manifold $M$ contains more information about its orbit polytope $P$ than the ring $H^\ast(\mathcal{Z}_P)$ does, but it does not determine the combinatorial type of $P$ in general. A simple polytope $P$ is called \emph{C-rigid} (over $k$) if its combinatorial type is determined by $H^\ast(M;k)$ for any quasitoric manifold $M$ over $P$---that is, if for every quasitoric manifold $M'$ over $P'$, a graded ring isomorphism $H^\ast(M;k)\cong H^\ast(M';k)$ implies a combinatorial equivalence $P\simeq P'$ (cf.\ \cite[Definition~1.2]{CPS}, \cite[Definition~A.11]{B}). The $3$-polytopes described in \cite[Example~3.4]{BEMPP} are neither $B$-rigid nor $C$-rigid, and examples of $C$-rigid polytopes that are not $B$-rigid are given in \cite{CP19}.

Although it follows from Theorem~\ref{main_decomp_th} and Lemma~\ref{M_Z_lem} that any product of $B$-rigid polytopes is $C$-rigid, it is not clear whether the class of $C$-rigid polytopes is closed under products. In this direction we have the following result:
\end{remark}

\begin{theorem} \label{quasitoric_thm1}
Let $M$ be a quasitoric manifold over a simple polytope $P$. If $P$ is indecomposable, then $H^\ast(M;k)$ cannot be decomposed as a nontrivial tensor product for any field $k$, and, in particular, $M$ is indecomposable up to homotopy.
\end{theorem}

\begin{proof}
We first note, with notation as above, that for any basis $z_1,\ldots,z_{m-n}$ of $H^2(M;k)$ there are isomorphisms
\begin{align*}
    H_*\big(\operatorname{Kos}^{{H^*(M;k)}}(z_1,\ldots,z_{m-n})\big)&\cong H_*\big( \operatorname{Kos}^{{k[P]}}(t_1,\ldots,t_{n},z_1,\ldots,z_{m-n})\big)\\
    &\cong   H_*\big(\operatorname{Kos}^{{k[P]}}(v_1,\ldots,v_m)\big)\\
    & \cong H_\ast(\mathcal{Z}_{P};k);
\end{align*}
the first isomorphism  
follows from \cite[Corollary 6.1.13 (b)]{BH}; the second isomorphism exists because $t_1,\ldots,t_{n},z_1,\ldots,z_{m-n}$ generate the same ideal as $u_1,\ldots,u_m$, and therefore have isomorphic Koszul complexes; the third isomorphism is Theorem \ref{t_H_Tor}.

Now suppose that $H^*(M;k)\cong A_1\otimes A_2$, with neither $A_1$ nor $A_2$ isomorphic to $k$. Take  bases $z_1,\ldots,z_i$ of $A_1^2$ and   $z_{i+1},\ldots,z_{m-n}$ of $A_2^2$. The K\"unneth isomorphism yields
\begin{align*}
H_*\big(\operatorname{Kos}^{A_1}(z_1,\ldots,z_{i})\big)\otimes H_*\big(&\operatorname{Kos}^{A_2}(z_{i+1},\ldots,z_{m-n})\big)
 \\  & \cong H_*\big( \operatorname{Kos}^{A_1\otimes A_2}(z_1,\ldots,z_{m-n})\big).
\end{align*}
We may use the chain of isomorphisms above to conclude that this is isomorphic to $H^\ast(\mathcal{Z}_{P};k)$, since $H^2(M;k)\cong A_1^2\oplus A_2^2$. Thereby we obtain a tensor product decomposition of $H^\ast(\mathcal{Z}_{P};k)$. By Corollary \ref{tfae_cor} the polytope $P$ decomposes accordingly into a product of polytopes.
\end{proof}

\begin{remark}
According to  Theorem~\ref{quasitoric_thm1},  if the cohomology $H^\ast(M)$ of a quasitoric manifold $M$ over $P$ admits a nontrivial tensor product decomposition, then the polytope $P$ decomposes as a product. We note that, unlike the case of moment-angle manifolds, the converse of this result is not true: many quasitoric manifolds over products of polytopes have indecomposable cohomology rings. For instance, the connected sum $\mathbb{C}P^3 \# \mathbb{C}P^3$ equipped with an appropriate $T^3$-action is a quasitoric manifold over the prism $\Delta^1\times \Delta^2$, while $H^\ast(\mathbb{C}P^3 \# \mathbb{C}P^3)$ admits no nontrivial tensor product decomposition.
\end{remark}
 
Since a simplex $\Delta^n$ is trivially $B$-rigid, Corollary~\ref{product_cor} implies that $B$-rigid polytopes are closed under products with simplices. We next consider a consequence of this fact for quasitoric manifolds and their orbit polytopes.

We first review a common method for constructing new quasitoric manifolds from a given one $M$. Let $E=\bigoplus_{i=0}^n \xi_i$ be a Whitney sum of complex line bundles over $M$. Removing the zero section and quotienting by the $\mathbb{C}^\ast$-action along each fibre, we obtain the projectivization $P(E)$ with fibre $\mathbb{C}P^n$. Starting with $M=\{ \mathrm{point} \}$, for example, and iterating this construction yields a tower of projective bundles 
\[ B_h \to B_{h-1} \to\cdots \to B_1 \to M=\{ \mathrm{point} \}, \]
called a \emph{generalized Bott tower}, where $B_j=P\left(\bigoplus_{i=0}^{n_j} \xi_i\right)$ for some complex line bundles $\xi_0,\ldots,\xi_{n_j}$ over $B_{j-1}$. In this case, the \emph{generalized Bott manifold} $B_h$ at height~$h$ is a quasitoric manifold over $\prod_{i=1}^h \Delta^{n_i}$. Note that since a product of simplices is $B$-rigid, generalized Bott manifolds provide a large class of quasitoric manifolds whose cohomology rings uniquely determine their orbit polytopes.

Let $G_m$ denote the $m$-gon, $m\geqslant 3$. In \cite{CP}, Choi and Park prove that the product of a simplex and a polygon is $B$-rigid and use this to obtain the following result. Following \cite{CP}, we call $P(E)$ a \emph{projective bundle} only when the vector bundle $E$ is a Whitney sum of complex line bundles, as above. 

\begin{theorem}[{\cite[Corollary~4.1]{CP}}]
Let $P(E)$ be a projective bundle over a $4$-dimensional quasitoric manifold. If the cohomology ring of a quasitoric manifold $M$ is isomorphic to that of $P(E)$, then the orbit space of $M$ is combinatorially equivalent to $\Delta^n \times G_m$.
\end{theorem}

Using Corollary~\ref{product_cor} we can extend the result above in two ways. First, since a $4$-dimensional quasitoric manifold is precisely a quasitoric manifold over a ($B$-rigid) polygon $G_m$, we can generalize to higher dimensions by replacing the role of the polygon with a $B$-rigid polytope $Q$ of arbitrary dimension. Second, we can consider towers of \emph{iterated} projective bundles starting with a quasitoric manifold over $Q$ since a $B$-rigid polytope remains so after taking a product with a simplex any number of times.

\begin{theorem} \label{quasitoric_thm2}
Let $Q$ be a $B$-rigid polytope and let $P(E)$ be an iterated projective bundle over any quasitoric manifold over $Q$. If the cohomology ring of a quasitoric manifold $M$ is isomorphic to that of $P(E)$, then the orbit space of $M$ is combinatorially equivalent to $\prod_{i=1}^h \Delta^{n_i} \times Q$.
\end{theorem}

\begin{proof}
Let $N^{2n}$ be a quasitoric manifold with $N^{2n}/T^n=Q$. First assume $P(E)$ is a projective bundle over $N^{2n}$. Then $E=\bigoplus_{i=0}^k \xi_i$ for some complex line bundles $\xi_0,\ldots,\xi_k$ over $N^{2n}$, so $E$ has a $T^{k+1}$-action defined by coordinatewise multiplication along each fibre. Moreover, since the fibre inclusion $\iota\colon N^{2n} \to ET^n \times_{T^n} N^{2n}$ of the Borel construction induces a surjection $\iota^\ast\colon k[Q] \to k[Q]/(t_1,\ldots,t_n)$, the first Chern class of each line bundle satisfies $c_1(\xi_j) \in \iota^\ast H^2_{T^n}(N^{2n};\mathbb{Z})$ and it follows from \cite[Theorem~1.1]{M} that the locally standard $T^n$-action on $N^{2n}$ lifts to a linear action on the total space of each $\xi_j$. The resulting action on $E$ makes $E\to N^{2n}$ a $T^n$-equivariant vector bundle, and since this action commutes with the $T^{k+1}$-action described above, these define an action of $T^{n+k+1}$ on $E$. The induced $T^{n+k}$-action on $P(E)$ is a locally standard half-dimensional torus action giving $P(E)$ the structure of a quasitoric manifold over $\Delta^k\times Q$, since the orbit space $\mathbb{C}P^k/T^k$ of the standard action on the fibre can be identified with $\Delta^k$ and $N^{2n}/T^n=Q$.

Next, assuming $P(E)$ is an iterated projective bundle of height $h$ over $N^{2n}$, iterating the argument above shows that $P(E)$ is a quasitoric manifold over a polytope $\prod_{i=1}^h \Delta^{n_i} \times Q$. Finally, if $M$ is a quasitoric manifold over some simple polytope $P$ with $H^\ast(M;k)\cong H^\ast(P(E);k)$, then by Lemma~\ref{M_Z_lem} there is an isomorphism of graded rings \[H^\ast(\mathcal{Z}_P;k) \cong H^\ast(\mathcal{Z}_{\prod_{i=1}^h \Delta^{n_i} \times Q};k).\] Since $\prod_{i=1}^h \Delta^{n_i} \times Q$ is $B$-rigid by Corollary \ref{product_cor}, $P$ is combinatorially equivalent to $\prod_{i=1}^h \Delta^{n_i} \times Q$.
\end{proof}


\begin{thebibliography}{99}

\bibitem{AG} L.\ L.\ Avramov, E.\ S.\ Golod, \emph{Homology algebra of the Koszul complex of a local Gorenstein ring}, Mat.\ Zametki,\ vol.\ 9, 1971, pp.\ 53--58; Math.\ Notes, vol.\ 9, 1971, pp.\ 30--32.

\bibitem{Bo} F.\ Bosio, \emph{Two transformations of simple polytopes preserving moment-angle manifolds}, \href{https://arxiv.org/abs/1708.00399}{\tt arXiv:1708.00399 [math.MG]}, 2017.

\bibitem{BH} W.\ Bruns and H.\ J.\ Herzog, \emph{Cohen--Macaulay Rings}, Cambridge Studies in Advanced Mathematics, vol.\ 39, Cambridge University Press, Cambridge, 1998.

\bibitem{B} V.\ M.\ Buchstaber, \emph{Lectures on toric topology}, Proceedings of Toric Topology Workshop: KAIST 2008, Trends in Math.\ \textbf{10}(1) (2008), 1--64.

\bibitem{BEMPP} V.\ M.\ Buchstaber, N.\ Y.\ Erokhovets, M.\ Masuda, T.\ E.\ Panov and S.\ Park, \emph{Cohomological rigidity of manifolds defined by right-angled $3$-dimensional polytopes}, Russian Math. Surveys \textbf{72}(2) (2017), 199--256.

\bibitem{BP} V.\ M.\ Buchstaber and T.\ E.\ Panov, \emph{Toric Topology}, Mathematical Surveys and Monographs, vol.\ 204, American Mathematical Society, Providence, RI, 2015.

\bibitem{Cai} L.\ Cai, \emph{On products in a real moment-angle manifold}, J.\ Math.\ Soc.\ Japan \textbf{69}(2) (2017), 503--528.

\bibitem{CK} S.\ Choi and J.\ S.\ Kim, \emph{Combinatorial rigidity of $3$-dimensional simplicial polytopes}, Int.\ Math.\ Res.\ Not.\ \textbf{8} (2011), 1935--1951.


\bibitem{CPS} S.\ Choi, T.\ E.\ Panov and D.\ Y.\ Suh, \emph{Toric cohomological rigidity of simple convex polytopes}, J.\ London Math.\ Soc.\ \textbf{82}(2) (2010), 343--360.

\bibitem{CP19} S.\ Choi and K.\ Park, \emph{Example of C-rigid polytopes which are not B-rigid}, Math.\ Slovaca \textbf{69}(2) (2019), 437--448.

\bibitem{CP} S.\ Choi and S.\ Park, \emph{Projective bundles over toric surfaces}, Internat.\ J.\ Math.\ \textbf{27}(4) (2016), 1650032.

\bibitem{DJ} M.\ W.\ Davis and T.\ Januszkiewicz, \emph{Convex polytopes, Coxeter orbifolds and torus actions}, Duke Math.\ J.\ \textbf{62}(2) (1991), 417--451.

\bibitem{E} N.\ Y.\ Erokhovets, \emph{B-rigidity of ideal almost Pogorelov polytopes}, Topology, Geometry, and Dynamics: V.\ A.\ Rokhlin-Memorial, A.\ V.\ M.\ Buchstaber, A.\ V.\ Malyutin, A.\ M.\ Vershik, eds.\ Contemp.\ Math. \textbf{772}, Amer.\ Math.\ Soc., Providence, RI, 2021, pp.\ 107--122.

\bibitem{FMW15} F.\ Fan, J.\ Ma and X.\ Wang, \emph{$B$-rigidity of flag $2$-spheres without $4$-belt}, \href{https://arxiv.org/abs/1511.03624}{\tt arXiv:1511.03624 [math.AT]}, 2015.

\bibitem{FMW20} F.\ Fan, J.\ Ma and X.\ Wang, \emph{Some rigidity problems in toric topology: I}, \href{https://arxiv.org/abs/2004.03362}{\tt arXiv:2004.03362 [math.AT]}, 2020.

\bibitem{Hochster} M.\ Hochster, \emph{Cohen--Macaulay rings, combinatorics, and simplicial complexes}, Ring Theory II (Proc.\ Second Oklahoma Conference).\ B.\ R.\ McDonald and R.\ Morris, eds.\
Dekker, NY, 1977, pp.\ 171--223.

\bibitem{MD}
M.\ Masuda, D.\ Y.\ Suh, \emph{Classification problems of toric manifolds via topology}, Toric Topology, M.\ Harada et al., eds.\ Contemp.\ Math. \textbf{460}, Amer.\ Math.\ Soc., Providence, RI, 2008, pp.\ 273--286. 

\bibitem{M} I.\ Mundet i Riera, \emph{Lifts of smooth group actions to line bundles}, Bull.\ London Math.\ Soc.\ \textbf{33} (2001), 351--361.

\bibitem{S} R.\ Stanley, \emph{Combinatorics and Commutative Algebra}, 2nd ed., Progress in Math., vol.\ 41, Birkh\"auser, Boston, MA, 1996.





\end{thebibliography}
\end{document}